\documentclass[a4paper,12pt]{amsart}
\usepackage[pagewise]{lineno}%\linenumbers
\usepackage[left=2cm,top=2cm,right=2cm,bottom=2cm]{geometry}
\usepackage{amsmath}
\usepackage{amssymb}
\usepackage{amsfonts}
\usepackage{amsthm}
\usepackage{stmaryrd}
\usepackage{hyperref}

\usepackage{amstext}
\usepackage{array}

\newtheorem{thm}{Theorem}[section]
\newtheorem{lemma}[thm]{Lemma}
\newtheorem{prop}[thm]{Proposition}

\newtheorem{prob}[thm]{Problem}

\theoremstyle{definition}

\theoremstyle{definition}
\newtheorem{defn}[thm]{Definition}

\DeclareMathOperator{\ord}{ord}

\DeclareMathOperator{\GL}{GL}
\DeclareMathOperator{\SL}{SL}
\newcommand{\PGL}{\ensuremath{\operatorname{PGL}}}

\newcommand{\supp}{\ensuremath{\operatorname{supp}}}
\newcommand{\Tr}{\ensuremath{\operatorname{Tr}}}
\newcommand{\Hom}{\ensuremath{\operatorname{Hom}}}

\newcommand{\adj}{\ensuremath{\operatorname{adj}}}

\newcommand{\norm}[1]{{\left\|{#1}\right\|}}

\newcommand{\abs}[1]{{\left|{#1}\right|}}
\newcommand{\scal}[1]{{\left\langle{#1}\right\rangle}}
\newcommand{\set}[1]{{\left\{{#1}\right\}}}

\newcommand{\be}{\ensuremath{\beta}}

\newcommand{\Ga}{\ensuremath{\Gamma}}
\newcommand{\de}{\ensuremath{\delta}}
\newcommand{\De}{\ensuremath{\Delta}}

\newcommand{\ka}{\ensuremath{\kappa}}
\newcommand{\La}{\ensuremath{\Lambda}}
\newcommand{\la}{\ensuremath{\lambda}}

\newcommand{\vphi}{\ensuremath{\varphi}}
\newcommand{\om}{\ensuremath{\omega}}

\newcommand{\bears}{\begin{eqnarray*}}
\newcommand{\eears}{\end{eqnarray*}}

\newcommand{\mc}[1]{\ensuremath{\mathcal{#1}}}

\newcommand{\ZZ}{\ensuremath{\mathbb{Z}}}

\newcommand{\Zp}{\ensuremath{\ZZ /p\ZZ }}

\newcommand{\FF}{\ensuremath{\mathbb{F}}}
\newcommand{\QQ}{\ensuremath{\mathbb{Q}}}
\newcommand{\RR}{\ensuremath{\mathbb{R}}}

\newcommand{\spark}{\ensuremath{\textnormal{sp}}}

\newcommand{\wt}[1]{\ensuremath{\widetilde{#1}}}

\newcommand{\ghg}{\ensuremath{G\times\widehat{G}}}

\newcommand{\CC}{\ensuremath{\mathbb{C}}}

\newcommand{\sm}{\ensuremath{\setminus}}
\newcommand{\ssq}{\ensuremath{\subseteq}}

\newcommand{\rar}{\ensuremath{\rightarrow}}
\newcommand{\longrar}{\ensuremath{\longrightarrow}}

\newcommand{\bs}[1]{\ensuremath{\mathbf{#1}}}
\providecommand{\abs}[1]{\lvert#1\rvert}
\newcommand{\bsm}{\ensuremath{\boldsymbol}}

\numberwithin{equation}{section}

\title[Spark deficient Gabor frames]{Spark deficient Gabor frames}% and uncertainty principles for the short--time Fourier transform}

\subjclass[2010]{42C15,15A03,11E95}

\keywords{Gabor frames, full spark, finite Weyl-Heisenberg groups, Clifford group, short-time Fourier transform, uncertainty principles}

\author{Romanos-Diogenes Malikiosis} 
\thanks{The author is supported by a Postdoctoral Fellowship from Humboldt Foundation.}
\address{Technische Universit\"at Berlin, Institut f\"ur Mathematik,
Sekretariat MA 4-1,
Stra{\ss}e des 17. Juni 136,
D-10623 Berlin, Germany}
\email{malikios@math.tu-berlin.de}

\begin{document}

\begin{abstract}
 The theory of Gabor frames of functions defined on finite abelian groups was initially developed in order to
 better understand the properties of Gabor frames of functions defined over the reals. However, during the last twenty
 years the topic has acquired an interest of its own. One of the fundamental questions asked in this finite setting is 
 the existence of full spark Gabor frames. The author proved the existence \cite {M15}, as well as constructed
 such frames, when the underlying group is finite cyclic. In this paper, we resolve the non-cyclic case; in particular, we
 show that there can be no full spark Gabor frames of windows defined on finite abelian non-cyclic groups.
 We also prove that all eigenvectors of certain unitary matrices in the Clifford group in
 odd dimensions generate spark deficient Gabor frames.
 Finally, similarities between the uncertainty principles concerning the finite dimensional Fourier
 transform and the short-time Fourier transform are discussed.
\end{abstract}

\maketitle

\bigskip
\bigskip
\section{Introduction}
\bigskip

The Gabor frame of a function $f\in L^2(\RR)$ is the set of all time-frequency translates of $f$, that is,
the set of all functions of the form $e^{2\pi ixy}f(x-t)$, for $y,t\in\RR$, and it is a fundamental concept
in time-frequency analysis and frame theory \cite{P13}. The function $f$ usually represents a signal, $t$ the time
delay, and the pointwise multiplication by $e^{2\pi ixy}$ is the frequency ``shift''. %Gabor frames are of fundamental importance in
%the theory of communication, signal and image processing, radar %%%GIVE CITATIONS
Through sampling and periodization \cite{C03} one passes to the finite version of a Gabor frame, namely the shift-frequency translates of
a complex function defined on a finite cyclic group. Even though finite dimensional Gabor frames were studied in order to analyze
the properties of continuous signals, they later developed an interest of their own.

Up to multiplication by roots of unity, a finite dimensional Gabor frame is the same as a \emph{Weyl-Heisenberg} orbit, and this terminology is much more prevalent in mathematical physics and
quantum information theory. A conjecture by Zauner \cite{Z99} states that for every dimension $N$ there are vectors (called ``fiducials'') whose WH orbit is equiangular. This means that the expression
$\abs{\scal{u,v}}$ is constant for every pair of distinct vectors $u$, $v$ within this orbit. 
This is also known as the SIC-POVM problem which has attracted a lot of attention lately due to the vast connections
to scientific areas such as quantum cryptography \cite{R05}, quantum tomography \cite{S06}, and algebraic number theory, especially Hilbert's 12th problem for real quadratic fields \cite{ABD15,AFMY16,AFMY17,AYAZ13}.
Such a WH orbit would then produce the maximal possible number of vectors in $\CC^N$ that are pairwise equiangular,
namely $N^2$ \cite{SH03}. Yet another terminology that appears for this phenomenon is \emph{maximal equiangular tight frame} (or maximal ETF for short) \cite{F09}, which is a special case
of the packing problem in the setting of projective spaces. The interest of the algebraic construction of families of ETFs has also increased due to applications to signal processing \cite{FMT12,IJM16,JMF14}.

A conjecture by Heil, Ramanathan, and Topiwala from 1996 \cite{HRT96} states that any finite set of a Gabor frame of
a nonzero $f\in L^2(\RR)$ is linearly independent, and it is still open. Similar questions can be raised when the function $f$ is defined on a finite abelian group $G$.
In this case, the Gabor frame consists of $\abs{G}^2$ elements in a $\abs{G}$-dimensional space, so it is not possible that they are linearly independent.
%Here, we will address the discrete
%version of this conjecture\footnote{In any case, we fall into either one of the following two
% extremes: either there are no full spark Gabor frames, or they do exist for almost all
% functions $f$, see Proposition \ref{extremes}.}%, whether there exist full spark Gabor frames, 
% when the function is
% defined on a finite abelian group, $G$. As this set consists of $\abs{G}^2$ elements in a $\abs{G}$-dimensional
% space, it is not possible that they are linearly independent. 
Instead, 
 we require that any selection of $\abs{G}$ vectors is linearly independent, which is the definition of the
 full spark property. 
 \begin{defn}
  Let $U=\set{u_1,\dotsc,u_M}\ssq\CC^N$ with $M\geq N$. The set $U$ is called \emph{full spark} when every selection of $N$ vectors from $U$ is linearly independent; otherwise, $U$ is called
  \emph{spark deficient}.
 \end{defn}
 The discrete analogue of the HRT conjecture claims that the Gabor frame of $f\in\CC^G$ is full spark for almost all $f$, when $G$ is cyclic \cite{LPW05}.
 This problem has been completely solved by the author \cite{M15}. While the techniques utilized to attack the HRT conjecture are analytic in nature, various algebraic techniques are needed for the
 discrete counterpart, such as Chebotarev's theorem on Fourier minors.
 The idea of the proof is as follows: consider $f$ as a column vector in $\CC^N$, where $\abs{G}=N$, and consider
 the $N\times N^2$ matrix whose columns are precisely the elements of the Gabor frame of $f$, denoted by $V_f$. % (this is also called the \emph{synthesis matrix} attached to $f$; its conjugate transpose
 %$V_f^*$ is called the \emph{short--time Fourier transform}). 
 The Gabor frame generated by $f$ is then full spark if and only if every $N\times N$ minor of $V_f$ is nonzero. Every such minor is 
 a homogeneous polynomial on the coordinates of $f$; the basic ingredient of the proof is to show that there is a monomial appearing with nonzero coefficient in every such minor. When $N$ is a prime, this was
 accomplished in \cite{LPW05} through Chebotarev's theorem, which asserts that every minor of the $N\times N$ discrete Fourier matrix is nonzero. When $N$ is composite, a probabilistic argument by the author 
 \cite{M15} was used in order to show the existence of monomials with nonzero coefficient in every minor. Furthermore, the author proved that almost every $f\in\CC^G$ generates a full spark Gabor frame, and
 explicitly constructed such frames, while previous proofs were only existential.
 
 For non-cyclic groups, it was previously only known that full spark Gabor frames \emph{do not exist} for
 functions defined on the Klein group, $\ZZ/2\ZZ\times\ZZ/2\ZZ$ \cite{P13}. We shall extend this argument to any finite abelian non-cyclic group,
 in the following way: first, we show that the full spark property is hereditary with respect to the group. Therefore, in order to show that
 no full spark Gabor frame exists, it suffices to restrict our attention to groups of the form $\ZZ/p\ZZ\times\ZZ/p\ZZ$, for $p$ odd prime.
 Thus is proved the first main result of this paper:
 \begin{thm}\label{noncyclic}
  Let $G$ be a finite abelian, non-cyclic group. Then, for any $f\in\CC^G$, the Gabor frame generated by $f$ is spark deficient.
 \end{thm}

 In relation to the SIC-POVM problem we will revisit the cyclic case and prove that all eigenvectors of Clifford unitaries whose (projective) order is not coprime to the dimension $N$, for $N$ odd,
 generate spark deficient Gabor frames, extending some results in \cite{DBBA13}. This shows that there is not in principle any relation between these two basic properties of a Gabor frame,
 namely equiangularity and the full spark property.
 
 Lastly, we investigate a possible connection between uncertainty principles with respect to the discrete and short-time Fourier transforms. Uncertainty principles provide a measure of localization
 of signals whose various transforms (e.g. Fourier) are well-localized. When these signals are defined over a finite Abelian group, localization is usually measured by the size of the support, leading
 to classical and new versions of uncertainty principles with respect to the Fourier transform \cite{Mes06,T05}. %%%%TAO, MESHULAM
 This sort of principle appears in applications to sparse signal recovery, and sparse matrix identification \cite{CRT06,KPR08,P07}, among others.
 
 The paper is organized as follows: in section \ref{BG}, we will give the definitions and the necessary backround related to the results of this paper. In
 section \ref{NC}, we will prove that full spark Gabor frames do not exist over finite abelian non-cyclic groups. Section \ref{CYC}
 revisits the cyclic case, where we find some special vectors that generate spark deficient Gabor frames, and section \ref{UNC}
 deals with uncertainty principles.

 \bigskip
\bigskip 
\section{Background}\label{BG}
\bigskip

\subsection{Notation}

Throughout this note, $G$ will denote a finite abelian group written additively,
and $\CC^G$ will denote the set of all complex valued functions defined on $G$. An element $f\in\CC^G$ will interchangeably be viewed as a vector in $\CC^N$, where $N=\abs{G}$, and as a
function $f:G\longrar \CC$. 
$\CC^N$ is equipped with an inner product $\scal{\cdot,\cdot}$, defined as follows:
\[\scal{x,y}=\sum_{i=1}^N x_i\bar{y}_i,\]
for $x=(x_1,\dotsc,x_N)$, $y=(y_1,\dotsc,y_N)$. Only in section \ref{8} will we use the bra-ket notation, $\scal{x|y}$, with the caution that complex conjugation
is taken on the coordinates of $x$. We remind that $|x\rangle$ denotes a column vector in $\CC^N$, and $\langle x|$ is its conjugate transpose; hence $|x\rangle\langle x|$ is the
$1$-dimensional projector onto $|x\rangle$.

Furthermore, we decided to use %the %cumbersome
$\ZZ/N\ZZ$ for the ring of residues $\bmod N$, and reserve $\ZZ_p$ for the ring of $p$-adic integers. Similarly, $\QQ_p$ denotes the field of $p$-adic
rational numbers.
%Furthermore, the size of the support of a function $f\in\CC^G$, i.e. the number of its nonzero values, is denoted by $\norm{f}_0$. 

For any $f\in\CC^N$ denote by $\widehat{f}$ the (unnormalized) Fourier transform of $f$; that is, $\widehat{f}=W_N f$, where 
$W_N=(\om^{ij})_{i,j=0}^{N-1}$, the character table of $\ZZ/N\ZZ$, with $\om=e^{2\pi i/N}$, and finally, let $\norm{f}_0$ denote the
cardinality of the support of $f$.

Two operators $U$ and $V$ on $\CC^N$ will be equal up to a phase if $U=e^{i\theta}V$; this will also be denoted as
\[U\stackrel{.}{=}V.\]
The \emph{projective} order of an operator $U$ is then defined to be the smallest nonnegative integer $m$ for which $U^m\stackrel{.}{=}I$.
Finally, the conjugate transpose of $U$ is denoted by $U^*$.
\bigskip

\subsection{Definitions}\label{defs}

 For any
$x\in G$ and $\xi\in\widehat{G}$, we define the operators $T_x, M_{\xi}:\CC^G\longrar\CC^G$, with
$T_x f(g)=f(g-x)$ and $M_{\xi}f(g)=\xi(g)f(g)$, for any $f\in\CC^G$, $g\in G$. The $T_x$ are called 
\emph{translation operators}, and the $M_{\xi}$ \emph{modulation operators}.
For any $\bsm \la=(x,\xi)\in G\times\widehat{G}$
the operators $\pi(\bsm \la)=M_{\xi}T_x$ are called \emph{time-frequency shift operators}. We have
\[M_{\xi}T_x=\xi(x)T_xM_{\xi},\]
or, in other words, $M_{\xi}$ and $T_x$ commute up to a phase. From this fact we get a faithful
projective representation
\[\rho:G\times\widehat{G}\longrar \PGL(\CC^G),\]
which is also irreducible \cite{FKL09,P13}. 

For a subset $\La\ssq G\times\widehat{G}$ and $f\in\CC^G\sm\set{0}$, the set
\[(f,\La)=\set{\pi(\bsm \la)f\vert \bsm \la\in\La}\]
is called a \emph{Gabor system}; if it spans $\CC^G$, it is called a \emph{Gabor frame}. This certainly happens when 
$\La=G\times\widehat{G}$ due to the irreducibility of $\rho$; in this case, it is also called a \emph{Weyl-Heisenberg orbit}.

\begin{defn}
 A set $\Phi$ of $M$ vectors in $\CC^N$ is called a \emph{frame} if it spans $\CC^N$. %Thus, Theorem \ref{noncyclic} follows directly from the theorems below. 
 In this case, we must have $M\geq N$. The \emph{spark}
 of $\Phi$, denoted by $\spark(\Phi)$, is the size of the smallest linearly dependent subset of $\Phi$. %When $\spark(\Phi)=N+1$, we say that
 %$\Phi$ is \emph{full spark}, otherwise we call $\Phi$ \emph{spark deficient}.
\end{defn}

A frame $\Phi$ is full spark if and only if every set of $N$ elements of $\Phi$ is a basis, or equivalently $\spark(\Phi)=N+1$, otherwise it is spark deficient. 
Other definitions are also found in literature; for example, in this case we also say that the vectors of $\Phi$ are in \emph{general linear position}, 
or also that $\Phi$ possesses the \emph{Haar property} \cite{P13}.

\begin{defn}
 For a window $\vphi\in\CC^G$, $\abs{G}=N$, let $V_{\vphi}$ denote the $N\times N^2$ matrix whose columns are the shift-frequency translates of $\vphi$, also called the \emph{synthesis operator}.  
 The operator $V_{\vphi}^*:\CC^G\rar\CC^{\ghg}$ is called the \emph{analysis operator}, or \emph{the short-time
Fourier transform} with window $\vphi$, defined by 
\[V_{\vphi}^*f=(\scal{f,M_{\xi}T_x\vphi})_{(x,\xi)\in\ghg}.\]
\end{defn}

The term ``window'' makes much more sense in the continuous setting, whence it originated. In signal processing, one analyzes a signal $f\in L^2(\RR)$ by integrating against elements of a frame
(e.g. Gabor frames, wavelets, etc.) generated by a well-localized function $\vphi$. Typical examples of well-localized functions include functions supported on an interval
(thus examining the given function on a small \emph{window of time}), or with very fast decay, such
as Gaussian functions; it should be emphasized that Gabor himself first applied Gabor frames on Gaussian window functions \cite{G46,P13}.

This term carries on to the discrete setting as well, however, we should note that the terms ``window'', ``vector'', and ``function'' (defined over a finite Abelian group) are interchangeable in what follows.

\bigskip

\subsection{Gabor systems of $\abs{G}=N$ vectors}

The Gabor system $(f,\La)$ with $\abs{\La}=N$ is linearly independent if and
only if the determinant of the matrix whose columns consist of the coordinates of the vectors $\pi(\la)f$, 
$\la\in\La$, is nonzero. This matrix is denoted by $D_{\La}$, and is well-defined up to permutation of its
columns. The determinant is denoted by $P_{\La}=\det(D_{\La})$, and is well-defined up to a sign, so it makes
sense to ask whether $P_{\La}$ is nonzero or not.

The most important property of $P_{\La}$, however, is the fact that it is a homogeneous polynomial of degree $N$
in $N$ variables, when the coordinates of $f$ are viewed as independent variables. So, the existence of an
element $f$ such that $(f,\La)$ is linearly independent happens precisely when $P_{\La}$ is a nonzero polynomial.
Investigating the properties of these polynomials $P_{\La}$ sheds light on the existence of Gabor frames in general
linear position.

A first crucial observation regarding linear independence, comes from the following:
\begin{prop}\label{extremes}
 There is a full spark Gabor frame defined over $G$, if and only if, for every 
 $\La\ssq\ghg$ with $\abs{\La}=N$ there is an $f\in\CC^G$ such that $(f,\La)$ is linearly independent. Moreover, either all windows $\vphi\in\CC^G$
 generate spark deficient Gabor frames, or almost all windows generate full spark Gabor frames.
\end{prop}

\begin{proof}
 One direction follows from definition: if $(f,\ghg)$ is full spark, then obviously every
 Gabor system $(f,\La)$ is linearly independent, for $\abs{\La}=N$. On the other hand, if for every $\La\ssq\ghg$
 with $\abs{\La}=N$ there is some $f\in\CC^G$ such that $(f,\La)$ is linearly independent, this means that
 all such polynomials $P_{\La}$ are nonzero. The zero set of every such polynomial is of Lebesgue measure zero, and
 since they are finitely many, this yields that almost any $f\in\CC^G$ avoids the zero set of these polynomials,
 hence $(f,\ghg)$ is full spark.
 
 For the second part, we observe that if at least one of the polynomials $P_{\La}$ is zero, then all Gabor frames defined over $G$ are spark deficient. Otherwise, as we have already
 shown, almost all Gabor frames are full spark.
\end{proof}
\bigskip

\subsection{The Weyl-Heisenberg and Clifford groups}

We restrict our attention to cyclic groups $G=\ZZ/N\ZZ$ of odd order, for convenience, as the results of this subsection will only be used towards the construction
of spark deficient Gabor frames over cyclic groups.
The group generated by the translation and modulation operators is
\[\set{\om^k M^b T^a| a,b,k\in\ZZ/N\ZZ},\]
where $\om=e^{2\pi i/N}$, $T=T_1$ (see \ref{defs}) and $M$ is the operator with the property $Mf(g)=\om^gf(g)$ for all $g\in\ZZ/N\ZZ$ and $f\in\CC^N$, and is called
the \emph{Weyl-Heisenberg} group of $G$. Sometimes \cite{A05,DBBA13,Z99}, these representatives over the center are considered:
\[D_{\bsm \la}=\tau^{\la_1\la_2}T^{\la_1}M^{\la_2},\]
where $\bsm \la=(\la_1,\la_2)\in(\ZZ/N\ZZ)^2$, $\tau=\om^{\frac{N+1}{2}}$.

It is known that all irreducible projective representations of $(\ZZ/N\ZZ)^2$ of dimension $N$ are unitarily equivalent to $\rho$ \cite{W31} (see also Proposition 3.2 \cite{FKL09}). 
The normalizer of the Weyl-Heisenberg group in the group
of unitary matrices in $N$ dimensions is called the \emph{Clifford group}, denoted by $C(N)$. The quotient of $C(N)$ by the Weyl-Heisenberg group is isomorphic to
$\SL(2,\ZZ/N\ZZ)$, hence $\rho$ can be extended to a faithful irreducible projective representation of $(\ZZ/N\ZZ)^2\rtimes \SL(2,\ZZ/N\ZZ)$, which we shall also denote by $\rho$,
abusing notation. Restricting this representation to
the right factor, $\SL(2,\ZZ/N\ZZ)$, we get a projective representation $F\mapsto U_F$, for $F\in\SL(2,\ZZ/N\ZZ)$. The unitary matrices $U_F$ act on the Weyl-Heisenberg group by conjugation:
\[U_F D_{\bsm \la}U_F^*=D_{F\bsm \la}.\]
More precisely, the following is true:
\begin{thm}[Theorem 1 \cite{A05}, $N$ odd]\label{appleby}
 There exists a unique isomorphism
 \[f:(\ZZ/N\ZZ)^2\rtimes \SL(2,\ZZ/N\ZZ)\longrightarrow C(N)/I(N)\]
 with the property $UD_{\bsm \la}U^*=\om^{[\bsm \vphi,F\bsm \la]}D_{F\bsm \la}$ for any $U\in f(\vphi,F)$, where $I(N)$ is the center of $C(N)$, and 
 $[\bsm \vphi,\bsm \chi]=\vphi_2\chi_1-\vphi_1\chi_2$.
\end{thm}

This yields the following theorem:
\begin{thm}\label{uniquerep}
 For $N$ odd, there is a unique faithful irreducible projective representation of $(\ZZ/N\ZZ)^2\rtimes\SL(2,\ZZ/N\ZZ)$ of dimension $N$, up to unitary equivalence.
\end{thm}

\begin{proof}
 Let $\rho$ be the standard representation of $(\ZZ/N\ZZ)^2\rtimes\SL(2,\ZZ/N\ZZ)$ defined in the beginning of this subsection, and let $\wt{\rho}$ denote another representation of dimension $N$ with
 the same properties. By Weyl's theorem \cite{FKL09,W31}, we may assume without loss of generality that
 \[\rho|_{(\ZZ/N\ZZ)^2}=\wt{\rho}|_{(\ZZ/N\ZZ)^2}.\]
 Since the image of $\SL(2,\ZZ/N\ZZ)$ acts by conjugation on the image of $(\ZZ/N\ZZ)^2$, the image of $\wt{\rho}$ will also be contained in $C(N)$. According to
 Theorem \ref{appleby}, for any $F\in\SL(2,\ZZ/N\ZZ)$, $\rho(F)$ and $\wt{\rho}(F)$ should differ by an element of $\rho((\ZZ/N\ZZ)^2)$, that is
 \[\wt\rho(F)\stackrel{.}{=}D_{\bsm \vphi}U_F.\]
 We will investigate the possibilities of $\bsm \vphi$ when $F=S$ or $T$, the generators of $\SL(2,\ZZ/N\ZZ)$, 
 \[S=\begin{pmatrix}
      0 & -1\\
      1 & 0
     \end{pmatrix},\	\	
     T=\begin{pmatrix}
        1 & 1\\
        0 & 1
       \end{pmatrix},
\]
 which satisfy $S^2=(ST)^3=-I$, as well as when $F=-I$. Assume therefore, that
 \[\wt\rho(T)\stackrel{.}{=}D_{\bsm \chi}U_T,\	\wt\rho(S)\stackrel{.}{=}D_{\bsm \psi}U_S,\	\wt\rho(-I)\stackrel{.}{=}D_{\bsm \mu}U_{-I}.\]
 Since $\wt\rho(S)^2\stackrel{.}{=}\wt\rho(-I)$, we must have
 \[\wt\rho(-I)\stackrel{.}{=}D_{\bsm \mu}U_{-I}\stackrel{.}{=}(D_{\bsm \psi}U_S)^2=D_{\bsm \psi+S\bsm \psi}U_{-I},\]
 hence
 \[\bsm \mu=(I+S)\bsm \psi.\]
 On the other hand, $\wt\rho(-T)\stackrel{.}{=}\wt\rho(-I)\wt\rho(T)\stackrel{.}{=}\wt\rho(T)\wt\rho(-I)$, whence
 \[D_{\bsm \chi+T\bsm \mu}U_{-T}\stackrel{.}{=}\wt\rho(T)\wt\rho(-I)\stackrel{.}{=}\wt\rho(-I)\wt\rho(T)\stackrel{.}{=}D_{\bsm \mu-\bsm \chi}U_{-T},\]
 therefore
 \[2\bsm \chi=(I-T)\bsm \mu,\]
 thus
 \[2\bsm \chi=(I-T)(I+S)\bsm \psi.\]
 Now, let
 \[\bsm \la=-(I-S)^{-1}\bsm \psi=-2^{-1}(I+S)\bsm \psi,\]
 so that
 \[\bsm \chi=-(I-T)\bsm \la,\	\bsm \psi=-(I-S)\bsm \la.\]
 Then,
 \[D_{\bsm \la}(D_{\bsm \chi}U_T)D_{\bsm \la}^*\stackrel{.}{=}D_{\bsm \la+\bsm \chi -T\bsm \la}U_T=U_T\]
 and
  \[D_{\bsm \la}(D_{\bsm \psi}U_S)D_{\bsm \la}^*\stackrel{.}{=}D_{\bsm \la+\bsm \psi -S\bsm \la}U_S=U_S,\]
  while obviously $D_{\bsm \la}D_{\bsm \vphi}D_{\bsm \la}^*\stackrel{.}{=}D_{\bsm \vphi}$, thus proving that
  \[D_{\bsm \la}\wt\rho(\bsm \vphi,F)D_{\bsm \la}^*\stackrel{.}{=}\rho(\bsm \vphi,F),\]
  for all $(\bsm \vphi,F)\in\SL(2,\ZZ/N\ZZ)\rtimes(\ZZ/N\ZZ)^2$, or in other words, $\rho$ and $\wt\rho$ are unitarily equivalent, completing the proof.
\end{proof}

Another way to obtain such a representation is the following: let
\[N=p_1^{r_1}\dotsm p_s^{r_s}\]
be the prime factorization of $N$. By Chinese Remainder Theorem we obtain
\[(\ZZ/N\ZZ)^2\rtimes \SL(2,\ZZ/N\ZZ)\cong \prod_{i=1}^s (\ZZ/p_i^{r_i}\ZZ)^2\rtimes \SL(2,\ZZ/p_i^{r_i}\ZZ),\]
and let
\[\SL(2,\ZZ/N\ZZ)\ni F\mapsto (F_i)_{1\leq i\leq s}\in\prod_{i=1}^s\SL(2,\ZZ/p_i^{r_i}\ZZ)\]
be the natural map according to the isomorphism above; that is, $F_i$ is the matrix obtained by reducing the entries of $F \bmod p_i^{r_i}$.
Assuming that $V_i\cong \CC^{p_i^{r_i}}$ is the fathful irreducible projective representation of $(\ZZ/p_i^{r_i}\ZZ)^2\rtimes \SL(2,\ZZ/p_i^{r_i}\ZZ)$ constructed as above, then we also
see that $V_1\otimes V_2\otimes \dotsm \otimes V_s$ is also a faithful irreducible representation of $(\ZZ/N\ZZ)^2\rtimes \SL(2,\ZZ/N\ZZ)$ 
(Theorem 10 \cite{S77}), and hence unitarily equivalent to the
standard one. This shows that $U_F$ is, up to unitary equivalence, equal to the Kronecker product of the $U_{F_i}$, thus
\begin{equation}\label{trkron}
 \Tr U_{F}=\prod_{i=1}^s \Tr U_{F_i},
 \end{equation}
a fact also pointed out in \cite{DBBA13}.

 \bigskip
 \bigskip
\section{Gabor frames over non-cyclic groups}\label{NC}
\bigskip

First we show that the full spark property is hereditary.

\begin{lemma}\label{hered}
 Let $G$ be a finite abelian group and $H$ a subgroup, such that no windows defined on $H$ generate full spark Gabor frames.
 Then, there exist no windows defined on $G$ that generate full spark Gabor frames.
\end{lemma}

\begin{proof}
 By hypothesis, there exists a set of pairs $(h_i,\xi_i)\in H\times\widehat{H}$, $1\leq i\leq\abs{H}$, such that the vectors
 $M_{\xi_i}T_{h_i}\vphi$ are linearly dependent for any choice of $\vphi\in\CC^H$. Now, extend the characters $\xi_i$ to $G$
 in all possible ways. In this way, we obtain pairs in $\ghg$ of the form $(h,\xi)$, where $h=h_i$ and $\xi|_H=\xi_i$, for some
 $i$; the number of these pairs is exactly $\abs{G}$, as there are $\abs{G/H}$ ways to extend a character of $H$ to a 
 character of $G$.
 
 Next, consider an arbitrary window $\psi\in\CC^G$. Since the vectors $M_{\xi_i}T_{h_i}\psi|_H$ are linearly dependent on $\CC^H$, there
 is a nonzero vector $f\in\CC^H$ such that all inner products $\scal{M_{\xi_i}T_{h_i}\psi_H,\bar{f}}=0$. Denote by $F$ the unique window of $\CC^G$
 for which we have $F|_H=f$ and $\supp(F)\ssq H$ (also a nonzero window). Then, for all $i$ and $\xi\in\widehat{G}$ with $\xi|_H=\xi_i$ we have
 \[\scal{M_{\xi}T_{h_i}\psi,\bar{F}}=\sum_{g\in G}\xi(g)\psi(g-h_i)F(g)=\sum_{h\in H}\xi_i(h)\psi(h-h_i)f(h)=\scal{M_{\xi_i}T_{h_i}\psi_H,\bar{f}}=0,\]
 which shows that these $\abs{G}$ pairs $(h,\xi)\in\ghg$ always give linearly dependent vectors, as desired.
\end{proof}

Since we wish to prove that there exist no windows over any finite abelian non-cyclic group that generate full spark Gabor frames, it suffices to
do so for groups of the form $\ZZ/p\ZZ\times\ZZ/p\ZZ$, for $p$ prime, due to the fundamental theorem of finite abelian groups; if such a group is non-cyclic, then
it must have a subgroup of this form. Thus, Theorem \ref{noncyclic} follows directly from Theorem \ref{prime}, which establishes the result for groups of the form $\Zp\times\Zp$.

When $p=2$, this has already been proven, therefore by Lemma \ref{hered} we know that any window defined on a group containing a copy of the Klein group as a subgroup
cannot generate a full spark Gabor frame. We provide an alternative proof of this statement, more in line with the proof of Lemma \ref{hered}, which also gives us
an estimate on the minimum value of $\norm{V_{\vphi}^*f}_0$.%, where $\norm{h}_0$ denotes the size of the support of $h$.

\begin{thm}
 Let $G$ be a finite abelian group that has a subgroup isomorphic to the Klein $4$-group. Then, there are no
 Gabor frames $(f,G\times\widehat{G})$ in general linear position; furthermore, we have
 \[\min\norm{V_{\vphi}^*f}_0\leq N^2-3N/2.\]
\end{thm}

\begin{proof}
 Let $K$ be the subgroup of $G$ isomorphic to the Klein $4$-group. For $f\in\CC^G$ define $\bar{f}$ satisfying
 $\bar{f}(g)=\overline{f(g)}$ for all $g\in G$, and define on $\CC^G$ an inner product given by
 \[\scal{f,h}=\sum_{g\in G}f(g)\bar{h}(g).\]
 By standard character theory, there are three nontrivial
 characters on $K$, and each one of them extends to $N/4$ characters on $G$, where $N=\abs{G}$.
 In total, there are $3N/4$ characters on $G$ whose restriction on $K$ is nontrivial. 
 
 Let $\xi$ be such a character, and let $f\in\CC^G\sm\set{0}$ be arbitrary. Let $a\in K$ be such that 
 $\xi(a)=-1$; there are two such elements of $K$, and so we consider the Gabor system consisting of
 time-frequency translates of the form
 \[M_{\xi}T_af,\	\xi\text{ nontrivial on K, }a\in K\text{ with }\xi(a)=-1.\]
 This system has $3N/2>N$ elements; we will show that each one of them is orthogonal to $\bar{f}$, and therefore
 the full Gabor frame $(f,G\times\widehat{G})$ cannot be in general linear position. Indeed,
 \begin{eqnarray*}
 \scal{M_{\xi}T_af,\bar{f}} &=& \sum_{g\in G}\xi(g)f(g-a)f(g)\\
 &=& \sum_{g\in G}\xi(g+a)f(g)f(g+a)\\
 &=& \sum_{g\in G}\xi(g)\xi(a)f(g-a)f(g)\\
 &=& -\sum_{g\in G}\xi(g)f(g-a)f(g)\\
 &=& -\scal{M_{\xi}T_af,\bar{f}},
 \end{eqnarray*}
 so $\scal{M_{\xi}T_af,\bar{f}}=0$. This also shows that $\norm{V_f\bar{f}}_0\leq N^2-3N/2$, proving the second part of the Theorem.
\end{proof}

\begin{thm}\label{prime}
 There are no full spark Gabor frames over $G=\ZZ/p\ZZ\times\ZZ/p\ZZ$, for $p$ prime.
\end{thm}

\begin{proof}
 The case $p=2$ has already been proven, so we may assume that $p$ is odd. As in the previous two proofs, we consider an arbitrary window $z\in\CC^G$, and then try to
 find a nonzero vector that is orthogonal to at least $\abs{G}=p^2$ shift-frequency translates of $z$. In order to find this desirable set of translates, we arrange
 the coordinates of $z$ in an array; here, we identify $\ZZ/p\ZZ\times\ZZ/p\ZZ$ with the finite field $\FF_q$, $q=p^2$, and $\theta\in\FF_q\sm\FF_p$:
 \begin{equation}\label{array}
 \begin{array}{|c|c|c|c|}\hline
     z_0 & z_{\theta} & \cdots & z_{-\theta}\\
     z_1 & z_{\theta+1} & \cdots & z_{-\theta+1}\\
     \vdots & \vdots & \ddots & \vdots\\
     z_{-1} & z_{\theta-1} & \cdots & z_{-\theta-1}\\ \hline
    \end{array}.
 \end{equation}
 We denote this $p\times p$ matrix by $Z$.
 The column vectors in $\CC^{\FF_p}$ from left to right are denoted by $Z_0,Z_{\theta},\dotsc,Z_{-\theta}$, respectively, and similarly,
 the row vectors by $Z'_0,Z'_1,\dotsc,Z'_{p-1}$. Next, consider the
 vector $x\in\CC^{\FF_q}$ whose matrix representation is precisely $X=(\adj Z)^*$, where $\adj Z$ denotes the \emph{adjugate} matrix
 of $Z$; we denote its columns by $X_0,X_{\theta},\dotsc,X_{-\theta}$ and its rows by $X'_0,X'_1,\dotsc,X'_{p-1}$. 
 The vector $x$ could be zero, however this happens for a
 set of Lebesgue measure zero. In particular, $x$ is zero precisely when all the $(p-1)\times(p-1)$ minors of $Z$ are zero, but all of them are nonzero polynomials on the
 coordinates of $z$, which shows that for almost all choices of $z$, $x$ is nonzero. If we prove that the Gabor frames with windows $z$ possessing that property
 are spark deficient, then by
 Proposition \ref{extremes} we get that all Gabor frames over $G=\ZZ/p\ZZ\times\ZZ/p\ZZ$ are spark deficient.
 
 We have
 \[\det Z\cdot I=Z^T \bar{X}=Z\bar{X}^T.\]
 The $(a,b)$ entry of $Z^T\bar{X}$ is $\scal{Z_{a\theta},X_{b\theta}}$, and similarly for $Z\bar{X}^T$ is $\scal{Z'_a,X'_b}$. We thus obtain
 \begin{equation}\label{inner}
\scal{Z_{a\theta},X_{b\theta}}=\scal{Z'_a,X'_b}=\de_{ab}\det Z,
\end{equation}
for every $a,b\in\FF_p$,
where $\de_{ab}$ is the usual Kronecker delta. Then, for every $a\in\theta\FF_p^*$ and $\xi\in\widehat{\FF}_q$ with $\xi|_{\FF_p}=\bs1_{\FF_p}$, we get
due to \eqref{inner}
\[\scal{M_{\xi}T_{a}z,x}=\sum_{b\in\FF_p}\xi(b\theta)\scal{Z_{b\theta-a},X_{b\theta}}=0.\]
This number of shift-frequency translates is $p(p-1)$, and we have just established that $x$ is orthogonal to all of them. Furthermore, if $a\in\FF_p^*$ and 
$\xi\in\widehat{\FF}_q$ with $\xi|_{\FF_p}=\bs1_{\theta\FF_p}$, we also get due to \eqref{inner}
\[\scal{M_{\xi}T_{a}z,x}=\sum_{b\in\FF_p}\xi(b)\scal{Z'_{b-a},X'_b}=0.\]
So far we have $2p(p-1)>p^2$ translates of $z$ orthogonal to $x$, so this already takes care of the spark deficiency of any Gabor frame over $G$.
We will find more translates orthogonal to $x$; let's put $a=0$ and
$\xi\in\widehat{\FF}_q$ with $\xi|_{\FF_p}=\bs1_{\FF_p}$, but $\xi\neq\bs 1_{\FF_q}$. Then, again we have by \eqref{inner}
\[\scal{M_{\xi}z,x}=\sum_{b\in\FF_p}\xi(b\theta)\scal{Z_{b\theta},X_{b\theta}}=\det Z\sum_{b\in\FF_p}\xi(b\theta)=0,\]
since $\xi|_{\theta\FF_p}\neq\bs 1_{\theta\FF_p}$. This number of pairs is exactly $p-1$.

Next, we still consider $a=0$, but $\xi\in\widehat{\FF}_q$ satisfies with $\xi|_{\FF_p}\neq\bs1_{\FF_p}$ and $\xi|_{\theta\FF_p}=\bs 1_{\theta\FF_p}$. Then,
\[\scal{M_{\xi}z,x}=\sum_{b\in\FF_p}\xi(b)\scal{Z'_b,X'_b}=\det Z\sum_{b\in\FF_p}\xi(b)=0,\]
by \eqref{inner}, thus giving us another $p-1$ orthogonal shift-frequency translates of $z$ orthogonal to $x$. In total, there are $2(p+1)(p-1)=2p^2-2$ such translates,
thus concluding the proof.
\end{proof}

\bigskip
\bigskip

\section{Spark deficient Gabor frames over cyclic groups}\label{CYC}

\bigskip
%\bigskip

Here we revisit the cyclic case. As it has already been proven by the author \cite{M15}, almost all windows generate full spark Gabor frames, so the spark deficient Gabor frames
are generated by exceptional vectors. When the order of the group is an odd, square-free integer, then all eigenvectors of certain unitaries belonging to the Clifford
group generate spark deficient Gabor frames \cite{DBBA13}. The motivation behind this result in \cite{DBBA13} was to establish a connection between equiangularity of 
a Gabor frame (SIC-POVM existence) and full spark, if any. In $3$ dimensions, the family of SIC-POVMs generated by vectors of the form $(0,1,-e^{i\theta})$
is always spark deficient, and Lane Hughston \cite{H07} first established a connection between the linear dependencies that arise from this SIC-POVM
for $\theta=0$ or $2\pi/9$ and the 
inflection points of an elliptic curve.
In general, it was proven in \cite{DBBA13} that when $N$ is an odd, square-free integer divisible by $3$, all eigenvectors of the \emph{Zauner unitary matrix}
generate spark deficient Gabor frames. Zauner's conjecture \cite{Z99}
states that an eigenvector of this matrix generates a SIC-POVM, i. e. a maximal equiangular tight frame. If it is true, then for all odd, square-free dimensions, this equiangular
tight frame is not full spark. This is another example that showcases the difference between a nice algebraic property of a Gabor frame (full spark) and a nice geometric
one (equiangularity); it is not necessary that both of them can appear, even when the second one appears at all. For unit norm tight frames in general, this is further
explained in \cite{K15}; see also \cite{FMT12,JMF14} where an infinite family of spark deficient equiangular tight frames is constructed, of arbitrarily high dimension.

When $N$ is not divisible by $3$, it is not known whether this SIC-POVM is also full spark or not. For example, it is full spark when $N=8$ \cite{DBBA13}, being the first construction
at that time of a full spark Gabor frame in $8$ dimensions\footnote{Explicit construction of a full spark Gabor frame in every dimension was later shown by the author \cite{M15}.}.

Concerning the eigenvectors of other Clifford unitaries, they also generate spark deficient Gabor frames as long as the (projective) order of the matrix divides $N$. We will
extend the results of section 7 in \cite{DBBA13}, ``Generalisation to other symplectic unitaries'', to all odd dimensions $N$ and unitaries whose order is not coprime to $N$.

\begin{thm}\label{mainthm}
Let $N$ be an odd integer. Then, any eigenvector of the unitary $U_F$ generates a spark deficient Gabor frame, where $F\in\SL(2,\ZZ/N\ZZ)$
and $\gcd(\ord(F),N)>1$.
\end{thm}

This is a direct consequence of the following theorem from \cite{DBBA13}, slightly rephrased in order to accommodate the terminology of this paper,
with the simple observation that if $\ord(F)=n$ and $\gcd(n,N)=d>1$, then the eigenvectors of
$U_F$ are also eigenvectors of $e^{i\theta}U_F^{n/d}=e^{i\theta}U_{F^{n/d}}$ (the phase $e^{i\theta}$ is arbitrary), while $\ord(F^{n/d})=d>1$, hence
$\ord(F^{n/d})$ divides $N$.

We call $\bs x\in(\ZZ/N\ZZ)^2$ \emph{$F$-full}, if the vectors $\bs x,F\bs x,\dotsc,F^{n-1}\bs x$ are all distinct, where $F\in\SL(2,\ZZ/N\ZZ)$ and $n=\ord(F)$.

\begin{thm}[Theorem 5 \cite{DBBA13}, odd version]\label{dbba5}
 Let $N$ be an odd positive integer and $F\in\SL(2,\ZZ/N\ZZ)$, and let $n=\ord(F)$. Suppose
 \begin{enumerate}
  \item $n>1$.
  \item $n$ divides $N$.
  \item $\Tr U_F\neq0$.
  \item There exist $N$ distinct points in $(\ZZ/N\ZZ)^2$ that are $F$-full.
 \end{enumerate}
 Then all eigenvectors of $U_F$ generate spark deficient Gabor frames.
\end{thm}

Conditions (3) and (4) always hold when $N$ is odd, as the following two Lemmata show; this was proven in \cite{DBBA13} for $N$ odd square-free\footnote{See Lemmata
7 and 8 in \cite{DBBA13}}.

\begin{lemma}\label{lem7}
 Let $N$ be an odd positive integer. Let $F\in\SL(2,\ZZ/N\ZZ)$ be arbitrary. Then the number of $F$-full 
 points in $(\ZZ/N\ZZ)^2$ is $\geq N\vphi(N)$, where $\vphi$ is Euler's function.
\end{lemma}

\begin{lemma}\label{lem8}
 Let $N$ be an odd positive integer. Then $\abs{\Tr(U_F)}\geq1$ for all $F\in\SL(2,\ZZ/N\ZZ)$.
\end{lemma}

We will dedicate the rest of this section to the proofs of these two Lemmata. For basic facts about the field of $p$-adic numbers, $\QQ_p$ and
its algebraic extensions, we refer the reader to \cite{C86,N99}.

\bigskip
\subsection{Proof of Lemma \ref{lem7}}

Let $F\in\SL(2,\ZZ/N\ZZ)$ and $F_i\in\SL(2,\ZZ/p_i^{r_i}\ZZ)$ be the reduction of $F$ modulo $p_i^{r_i}$, $1\leq i\leq s$. Similarly, with $\bs x\in (\ZZ/N\ZZ)^2$ and
$\bs x_i\in\SL(2,\ZZ/p_i^{r_i}\ZZ)$. It is not hard to show that if each $\bs x_i$ is $F_i$-full, then $\bs x$ is $F$-full, a fact also shown in \cite{DBBA13}.
By multiplicativity of the Euler function, it suffices to consider $N=p^r$, a power
of an odd prime.

%The property described in the statement of the lemma is multiplicative in nature, because the symplectic group
%$\SL(2,\ZZ/N\ZZ)$ splits into the components $\SL(2,\ZZ/p_j^{r_j}\ZZ)$, $1\leq j\leq k$, and the function
%$N\vphi(N)$ is also multiplicative (in other words, the argument on the top of page 21 in \cite{DBBA13}
%applies to the more general case, where $N$ is any odd number). So, it suffices to consider $N=p^r$, a power
%of an odd prime.

The case $r=1$ was treated in \cite{DBBA13}. The technique was to find the Jordan canonical form of $F$, considering
a quadratic extension of the field $\ZZ/p\ZZ$ if necessary (i. e. $\FF_{p^2}$); then we can control the powers
of $F$ and can count the points in $(\ZZ/p\ZZ)^2$ that are $F$-full.

When $r>1$, $\ZZ/N\ZZ$ is no longer a field, so the Jordan canonical form does not always exist, but as we
shall see below, in these exceptional cases, the order of $F$ is equal to $p^m$ or $2p^m$, 
for some $m\leq r$, so we only need to enumerate
the points in $(\ZZ/N\ZZ)^2$ that are fixed by $F^{p^{m-1}}$ or $F^{2p^{m-1}}$ accordingly, 
and as it turns out, this is an easy task.

It would be convenient to consider an arbitrary lift of the matrix 
\[F=\begin{pmatrix}
          a & b\\
          c & d
         \end{pmatrix}
\]
to a matrix in $\wt{F}\in\SL(2,\ZZ_p)$; since $F\in\SL(2,\ZZ/N\ZZ)$, then at least one of the entries
$a$, $b$, $c$, $d$, is not divisible by $p$, say $a$. Then, lift $a$, $b$, $c$, arbitrarily, to $\wt{a}$, $\wt{b}$,
$\wt{c}$, and put $\wt{d}=\wt{a}^{-1}(1+\wt{b}\wt{c})$. We also put $t=\Tr(F)$, $\wt{t}=\Tr(\wt{F})$,
$\De=t^2-4$, $\wt{\De}=\wt{t}^2-4$, the discriminants of the characteristic polynomials of $F$, $\wt{F}$,
respectively. Finally, we put
\[\la=\frac{t+\sqrt{\wt{\De}}}{2},\]
and the other root of the characteristic polynomial is $\la^{-1}$. We distinguish the following cases:

\bigskip
\noindent
$\boxed{p\nmid \De}$ In this case, $\la\not\equiv\la^{-1}\bmod p$; otherwise, we would have
\[\De\equiv(\la+\la^{-1})^2-4\equiv\la^2+\la^{-2}-2\equiv0\bmod p.\]
We reduce the entries of $F\bmod p$. Since $\la\not\equiv\la^{-1}\bmod p$, $F$ is
diagonalizable in $\Zp$ when $(\frac{\De}{p})=1$ or in a quadratic extension, namely $\FF_{p^2}$, when
$(\frac{\De}{p})=-1$. In both cases, we consider the field $K=\QQ_p(\sqrt{\wt{\De}})$, whose ring of integers
is $\mc{O}_K=\ZZ_p[\sqrt{\wt{\De}}]$ and the unique prime ideal is $p\mc{O}_K=p\ZZ_p[\sqrt{\wt{\De}}]$. This extension
is unramified, as $p\nmid\De$, hence the degree of the extension is equal to the degree of the extension of the residue fields. Therefore,
the residue field of $K$ is $\FF_p$ when $(\frac{\De}{p})=1$ and $\FF_{p^2}$ otherwise.

So, there is a nonsingular matrix $X$ with entries in  the residue field of $K$ such that
\begin{equation}\label{jordan}
FX\equiv X\begin{pmatrix}
            \la & 0\\
            0 & \la^{-1}
           \end{pmatrix}\bmod p\mc{O}_K,
\end{equation}
the congruence meaning that we consider each entry $\bmod p\mc{O}_K$.
We can lift
$X=\begin{pmatrix}
    x & z\\
    y & w
   \end{pmatrix}
$
to a $2\times2$ matrix with entries in $\mc{O}_K$, such that \eqref{jordan} becomes an
equality in $\mc{O}_K$ (and holds $\bmod N$, in particular). Indeed, if $b$ is not divisible by $p$, then we lift
$x$, $z$ arbitrarily, and then put $y=\wt{b}^{-1}(\la-\wt{a}x)$, $w=\wt{b}^{-1}(\la-\wt{a}z)$, and a similar
lift is possible if $c$ is not divisible by $p$. If both $b$ and $c$ are divisible by $p$, then $F\bmod p$ is
diagonal, hence 
$F\equiv\begin{pmatrix}
            \la & 0\\
            0 & \la^{-1}
           \end{pmatrix}$ or $\begin{pmatrix}
            \la^{-1} & 0\\
            0 & \la
           \end{pmatrix}\bmod p$. Without loss of generality, we may assume that the first congruence holds. Lift
$x$, $w$, arbitrarily, and then put $y=(\la-\la^{-1})^{-1}\wt{c}x$, and $z=(\la^{-1}-\la)^{-1}\wt{b}w$. We notice
that since $p\nmid\det(X)$, then $X^{-1}\in\GL(2,\mc{O}_K)$; we conclude that in all cases where $p\nmid\De$, $F$
is equivalent to a diagonal matrix, with entries perhaps in a larger ring. It is evident that in this case, the number
of $F$-full points is $N^2-1$, since $p\nmid\la$, and $\la\not\equiv1\bmod p$.

\bigskip
\noindent
$\boxed{p\mid\De}$ Reducing the matrix $F\bmod p$, we obtain a double eigenvalue, equal to $\pm1$. Then, the
Jordan canonical form of $F$ is
\[\begin{pmatrix}
   \pm1 & \be\\
   0 & \pm1
  \end{pmatrix}
\]
where $\be=0$ or $\be=1$.
It is clear that $F^p\equiv\pm I\bmod p$ and $F^{2p}\equiv I\bmod p$, or $F^{2p}\equiv I+pA\bmod{p^2}$,
for some matrix $A$. Raising both sides to the $p$-th power, we get $F^{2p^2}\equiv I+p^2A\bmod{p^3}$, 
and proceeding inductively we can show that
\[F^{2p^{r-1}}= I+\frac{N}{p}A,\]
hence $F^{2N}=I$. This shows that the order of $F$ is either $p^m$ or $2p^m$, for
some $m\leq r$.

Suppose first that the order of $F$ is $p^m$ ($m\geq1$); then, an element of $(\ZZ/N\ZZ)^2$ is $F$-full,
if and only if
it is not fixed by $F^{p^{m-1}}$ (this follows from the fact that the cardinality of the orbit of any
element under a group, divides the order of the group), and the latter is equivalent to the condition that
this element is $F^{p^{m-1}}$-full. Therefore, we can reduce to the case where $m=1$, that is, the order
of $F$ is $p$. Since the number of $F$-full points is the same in the conjugacy class of $F$, we
may further assume that $F$ reduced $\bmod p$ is equal to $\begin{pmatrix}
   \pm1 & \be\\
   0 & \pm1
  \end{pmatrix}$. Now, let $k$ be the smallest positive integer for which we have
\[F\equiv\begin{pmatrix}
               1 & \be\\
               0 & 1
              \end{pmatrix}+p^{k-1}D\bmod{p^k}
\]
for some matrix $D\not\equiv\bs{O}\bmod p$, where $\bs{O}$ is the zero matrix. We have $2\leq k\leq r+1$. If $\be=0$, then $k=r$; if $k<r$, then
\[F^p\equiv I+p^k D\bmod{p^{k+1}},
\]
hence $F^p\neq I$, a contradiction. Similarly, if $k=r+1$, then $F=I$, which is also a contradiction.
So, $F=I+\frac{N}{p}D$.
A vector $\bs{x}=\begin{pmatrix}
                                                  \bs{x}_1\\
                                                  \bs{x}_2
                                                 \end{pmatrix}\in(\ZZ/N\ZZ)^2
$ is fixed by $F$ if and only if
\[D\bs{x}\equiv\bs{0}\bmod p.\]
The set of such vectors reduced $\bmod p$ form a proper vector subspace of $(\Zp)^2$, so they are at most $p$. Then,
the number of all the possible lifts of these vectors is at most $\bmod N$ is $p^{2(r-1)}\cdot p=p^{2r-1}$. Therefore,
the number of $F$-full vectors in this case is at least $p^{2r}-p^{2r-1}=N\vphi(N)$.

If $\be=1$, then
\[F^p\equiv\begin{pmatrix}
               1 & p\\
               0 & 1
              \end{pmatrix}+p^{k-1}\sum_{\ka+\mu=p-1}
              \begin{pmatrix}
               1 & \ka\\
               0 & 1
              \end{pmatrix}
D\begin{pmatrix}
  1 & \mu\\
  0 & 1
 \end{pmatrix}
\bmod{p^k}.
              \]
We put $D=\begin{pmatrix}
           d_1 & d_2\\
           d_3 & d_4
          \end{pmatrix}
$ and compute the above sum $\bmod p$:
\begin{eqnarray*}
 \sum_{\ka+\mu=p-1}
              \begin{pmatrix}
               1 & \ka\\
               0 & 1
              \end{pmatrix}
D\begin{pmatrix}
  1 & \mu\\
  0 & 1
 \end{pmatrix} &=& \sum_{\ka+\mu=p-1}\begin{pmatrix}
		    d_1+\ka d_3 & \mu d_1+\ka\mu d_3+d_2+\ka d_4\\
		    d_3 & \mu d_3+d_4
		    \end{pmatrix}\\
		    &\equiv& \bs{O} \bmod p
\end{eqnarray*}
since 
\[\sum_{\ka+\mu=p-1}1=p, \sum_{\ka+\mu=p-1}\ka=\sum_{\ka+\mu=p-1}\mu=p\cdot\frac{p-1}{2}, 
\sum_{\ka+\mu=p-1}\ka\mu=p\cdot\left(\frac{(p-1)^2}{2}-\frac{(p-1)(2p-1)}{6}\right).\]
But then, $F^p\not\equiv I\bmod{p^k}$, a contradiction if $k\leq r$; if $k=r+1$, then
$F^p=\begin{pmatrix}
           1 & p\\
           0 & 1
          \end{pmatrix}\neq I
$. We conclude that if the order of $F$ is $p$ and $r\geq2$, then $\be=0$ (the case $\be\neq0$ can only occur
when $r=1$, but this was treated in \cite{DBBA13}).

Next, suppose that the order of $F$ is $2p^m$. Then, a vector is $F$-full
if and only if it is not fixed by $F^{p^m}$ or $F^{2p^{m-1}}$. But $F^{p^m}=-I$, 
which only fixes the zero vector, so
we only need to exclude the vectors fixed by $F^{2p^{m-1}}$; however, this matrix has order $p$, so the above 
analysis applied to $F^{2p^{m-1}}$ yields the fact that the number of $F$-full points is at 
least $N\vphi(N)$.

\bigskip

\subsection{Proof of Lemma \ref{lem8}}\label{8}

The trace $\Tr(U_{F})$ is a quadratic Gauss sum \cite{A05}; we will use the following lemma by Turaev \cite{T98}
(see Lemma 1) which gives the 
absolute value of such a sum over an arbitrary finite abelian group $G$. Moreover, by \eqref{trkron} we may assume
that $N$ is a power of an odd prime, $p$.

Let's fix some notation first; 
$q:G\longrightarrow\QQ/\ZZ$ denotes an arbitrary quadratic form on the finite abelian group $G$. Such a function
is a quadratic form if the expression $b^q(x,y)=q(x+y)-q(x)-q(y)$ is bilinear (\emph{we do not require homogeneity}).
The Gauss sum $\Ga(G,q)$ is defined to be
\[\frac{1}{\abs{G}^{1/2}}\sum_{x\in G}e^{2\pi iq(x)}.\]
Lastly, for easy reference to the explicit formulae for the unitary matrices $U_F$ given in \cite{A05}, we decided to use the
bra-ket notation; the set of (column) vectors
\[|0\rangle,|1\rangle.\dotsc,|N-1\rangle,\]
is the standard basis of $\CC^N$, and $\langle\vphi|$ is the conjugate transpose of $|\vphi\rangle$.

\begin{lemma}[Lemma 1 \cite{T98}]
 Let $B$ be the kernel of the homomorphism $G\longrightarrow\Hom(G,\QQ/\ZZ)$ adjoint to the pairing $b^q$. If $q(B)\neq0$,
 then $\Ga(G,q)=0$. If $q(B)=0$, then $\abs{\Ga(G,q)}=\abs{B}^{1/2}$.
\end{lemma}

If $p\nmid b$, then the matrix $F$ is called \emph{prime}, and from the explicit formulae of \cite{A05} 
(see Lemma 2 and Lemma 4), we get
\[U_{F}=\frac{e^{i\theta}}{\sqrt{N}}\sum_{r,s=0}^{N-1}\tau^{b^{-1}(as^2-2rs+dr^2)}|r\rangle\langle s|,\]
where $\theta$ is an arbitrary phase, and $b^{-1}$ the inverse of $b\bmod N$, hence
\[\Tr(U_{F})=\frac{e^{i\theta}}{\sqrt{N}}\sum_{r=0}^{N-1}\tau^{b^{-1}(t-2)r^2}.\]
where $\tau=-e^{\frac{\pi i}{N}}$ and $t=a+d=\Tr(F)$. Putting $G=\ZZ/N\ZZ$ and
\[q(r)=\frac{b^{-1}(t-2)(N+1)}{2N}r^2,\]
we get $\Tr(U_{F})=e^{i\theta}\Ga(G,q)$.
$q$ is a well-defined quadratic form on $G$; indeed, as $r^2\equiv r'^2\mod{2N}$, when $r\equiv r'\mod N$, when
$N$ is odd. The associated bilinear pairing is
\[b^q(r,s)=\frac{b^{-1}(t-2)(N+1)}{N}rs,\]
and $r\in B$ if and only if $b^q(r,1)=0$, or equivalently, if
\[b^{-1}(t-2)r\equiv0\bmod N.\]
So, if $r\in B$ is arbitrary, then $N$ divides $b^{-1}(t-2)r^2$, hence $2N$ divides $b^{-1}(t-2)(N+1)r^2$, which
shows that $q(r)=0$. This proves that $q(B)=0$, hence $\abs{\Ga(G,q)}=\abs{B}^{1/2}\geq1$ and 
$\abs{\Tr(U_{F})}\geq1$.

Now, assume that $p\mid b$; then $p\nmid d$ (otherwise $\det(F)$ would be divisible by $p$)
and we can write $F$ as a product of two prime matrices, as follows:
\[F=F_1F_2=\begin{pmatrix}
                           0 & -1\\
                           1 & 0
                          \end{pmatrix}\begin{pmatrix}
                          c & d\\
                          -a & -b
                          \end{pmatrix}
\]
and by Lemma 4 \cite{A05}, we have $U_{F}=U_{F_1}U_{F_2}$, where
\[U_{F_1}=\frac{e^{i\theta_1}}{\sqrt{N}}\sum_{u,v=0}^{N-1}\tau^{2uv}|u\rangle\langle v|\]
and
\[U_{F_2}=\frac{e^{i\theta_2}}{\sqrt{N}}\sum_{v,w=0}^{N-1}\tau^{d^{-1}(cw^2-2vw-bv^2)}|v\rangle\langle w|,\]
where $\theta_1$, $\theta_2$ arbitrary phases, hence
\[U_{F}=\frac{e^{i\theta}}{N}\sum_{u,w=0}^{N-1}\sum_{v=0}^{N-1}\tau^{2uv+d^{-1}(cw^2-2vw-bv^2)}|u\rangle\langle w|\]
and
\[\Tr(U_{F})=\frac{e^{i\theta}}{N}\sum_{u,v=0}^{N-1}\tau^{cd^{-1}u^2+2(1-d^{-1})uv-bd^{-1}v^2},\]
where $\theta=\theta_1+\theta_2$.
So, if we put $G=(\ZZ/N\ZZ)^2$ and $q:G\longrightarrow\QQ/\ZZ$ the quadratic form
\[q(u,v)=\frac{N+1}{2N}(cd^{-1}u^2+2(1-d^{-1})uv-bd^{-1}v^2)\]
then $\Tr(U_{F})=e^{i\theta}\Ga(G,q)$. The associated bilinear form is
\[b^q((u,v),(r,s))=\frac{N+1}{N}(u\ v)A\begin{pmatrix}
                                        r\\
                                        s
                                       \end{pmatrix}
\]
where
\[A=\begin{pmatrix}
     cd^{-1} & 1-d^{-1}\\
     1-d^{-1} & -bd^{-1}
    \end{pmatrix}.
\]
Now, let $(u\ v)\in B$ be arbitrary. Then,
\[(u\ v)A\equiv(0\ 0)\bmod N,
\]
otherwise, we would have either $b^q((u,v),(1,0))\neq0$ or $b^q((u,v),(0,1))\neq0$. In particular, $N$ divides
$b^q((u,v),(u,v))$, and since $N$ is odd, $2N$ divides
\[(N+1)(u\ v)A\begin{pmatrix}
                                        u\\
                                        v
                                       \end{pmatrix}
\]
which yields $q(u,v)=0$. Thus, $q(B)=0$, and $\abs{\Tr(U_{F})}=\abs{\Ga(G,q)}=\abs{B}^{1/2}\geq1$.

\bigskip
\bigskip

\section{Uncertainty principles}\label{UNC}
\bigskip

\noindent
The full spark property of (almost all) Gabor frames of windows defined over finite cyclic groups implies the following inequality for the short-time Fourier
transform of $f$:
\[\norm{V_{\vphi}^*f}_0\geq N^2-N+1,\]
where $N$ is the size of said group, for almost all $\vphi\in\CC^N$ and all nonzero $f\in\CC^N$ \cite{KPR08,M15,P13}. A possible connection between the set of pairs
of the form $(\norm{f}_0,\bigl\|\widehat{f}\bigr\|_0)$, denoted by $F$, and the set $F_{\vphi}$ of all pairs of the form $(\norm{f}_0,\norm{V_{\vphi}^*f}_0-N^2+N)$
(for both sets we take $f$ nonzero) was investigated in \cite{KPR08}. In particular, the following problem was proposed.

\begin{prob}[\cite{KPR08}]\label{kpr}
 Is it true that $F=F_{\vphi}$ for almost all $\vphi$?
\end{prob}
When $N=p$ a prime number, this problem was solved to the affirmative \cite{KPR08}. One has an exact characterization of the set $F$ \cite{T05} and the fact
that all minors of the Gabor synthesis matrix are nonzero for all $\vphi$ except for a set of measure zero (Theorem 4 \cite{LPW05}), 
leads to a characterization of the set $F_{\vphi}$,
and equality between $F$ and $F_{\vphi}$ is easily confirmed. When $N$ is composite, however, there is no exact characterization for the set $F$, so it is more
difficult to obtain equality; this was confirmed numerically for dimensions up to $6$ \cite{KPR08}. The question is whether we can prove equality between those
two sets without using the characterization of $F$. We will show that one inclusion is possible, but the other one, namely $F_{\vphi}\ssq F$ seems much harder to
prove, if true.

As a final remark, we note that the spark deficiency of all Gabor frames of windows defined over abelian, non-cyclic groups, implies that equality between
$F$ and $F_{\vphi}$ can never be achieved, simply because there are $f\in\CC^G$ for which $\norm{V_{\vphi}^*f}_0\leq N^2-N$, as shown in the proof of Theorem 
\ref{prime}.

A useful identity is the following:
\begin{equation}\label{supportSTFT}
\norm{V_{\vphi}^*f}_0=\sum_{j=0}^{N-1}\norm{\widehat{T^j\vphi\cdot f}}_0.
\end{equation}

\begin{thm}
For almost all $\vphi$ the inclusion
$F\ssq F_{\vphi}$ holds. In addition, this $\vphi$ can be taken to generate a full spark Gabor frame.
\end{thm}

\begin{proof}
 First, we may restrict our attention to $\vphi$ generating a full spark Gabor frame, as we already know that almost all $\vphi$ satisfy
 this condition. This implies that all coordinates of $\vphi$ are nonzero, otherwise the frequency translates of $\vphi$ would form
 a singular matrix.
 Next, for any pair $(k,l)\in F$ we consider $f_{k,l}\in\CC^N$ with $\norm{f_{k,l}}_0=k$ and $\bigl\|\widehat{f_{k,l}}\bigr\|_0=l$. We may rewrite \eqref{supportSTFT}
 as
 \begin{equation}\label{supportSTFT2}
\norm{V_{\vphi}^*\frac{f_{k,l}}{\vphi}}_0=\sum_{j=0}^{N-1}\norm{\widehat{\frac{T^j\vphi}{\vphi}\cdot f_{k,l}}}_0
=\norm{\widehat{f_{k,l}}}_0+\sum_{j=1}^{N-1}\norm{\widehat{\frac{T^j\vphi}{\vphi}\cdot f_{k,l}}}_0=
l+\sum_{j=1}^{N-1}\norm{\widehat{\frac{T^j\vphi}{\vphi}\cdot f_{k,l}}}_0.
\end{equation}
It suffices to show that almost all $\vphi$ satisfy
\[\norm{\widehat{\frac{T^j\vphi}{\vphi}\cdot f_{k,l}}}_0=N,\]
for all $(k,l)\in F$ and $1\leq j\leq N-1$, or equivalently, it suffices to show that
\[\Phi\sum_{g=0}^{N-1}\xi(g)f_{k,l}(g)\frac{\vphi(g-j)}{\vphi(g)}\neq0,\]
for almost all $\vphi\in\CC^N$, all characters $\xi$, $(k,l)\in F$, $1\leq j\leq N-1$, where $\Phi$ is the product of the coordinates of $\vphi$. 
But the left-hand side is a polynomial in the coordinates of $\vphi$ with coefficients of the form $\xi(g)f_{k,l}(g)$, which shows that every such polynomial is nonzero, as
the functions $f_{k,l}$ are not identically zero. Therefore, $\vphi$ has to avoid the zero set of finitely many nonzero polynomials, whose union is of measure zero. Thus,
almost all $\vphi$ satisfy
\[\norm{V_{\vphi}^*\frac{f_{k,l}}{\vphi}}_0=N^2-N+l,\]
for every $(k,l)\in F$, as desired.
\end{proof}


\bigskip
 \bigskip
 
 \begin{thebibliography}{}

\bibitem{A05}
D. M. Appleby.
\newblock ``SIC-POVMs and the extended Clifford group''.
\newblock {\em J. Math. Phys.} {\bfseries 46} 052107 (2005).

\bibitem{ABD15}
D. M. Appleby, I. Bengtsson, H. B. Dang.
\newblock ``Galois unitaries, mutually unbiased bases, and MUB-balanced states''. 
\newblock {\em Quantum Inf. Comput.} {\bfseries 15}, no. 15-16, 1261--1294 (2015).

\bibitem{AFMY16}
M. Appleby, S. Flammia, G. McConnell, J. Yard.
\newblock ``Generating Ray Class Fields of Real Quadratic Fields via Complex Equiangular Lines''.
\newblock {\em ArXiv preprint}, \url{https://arxiv.org/abs/1604.06098} (2016).


\bibitem{AFMY17}
M. Appleby, S. Flammia, G. McConnell, J. Yard. 
\newblock ``SICs and Algebraic Number Theory''. 
\newblock {\em Found. Phys.} {\bfseries 47}, no. 8, 1042--1059 (2017).

\bibitem{AYAZ13}
D. M. Appleby, H. Yadsan-Appleby, G. Zauner.
\newblock ``Galois automorphisms of a symmetric measurement''.
\newblock {\em Quantum Inf. Comput.} {\bfseries 13}, no. 7-8, 672--720 (2013).

\bibitem{CRT06}
E. J. Cand\'{e}s, J. Romberg, T. Tao.
\newblock ``Robust uncertainty principles: exact signal reconstruction from highly incomplete frequency information''. 
\newblock {\em IEEE Trans. Inform. Theory} {\bfseries 52}, no. 2, 489--509 (2006).

\bibitem{C86}
J. W. S. Cassels.
\newblock ``Local fields.''
\newblock London Mathematical Society Student Texts, 3. {\em Cambridge University Press, Cambridge}, xiv+360 pp. (1986).

\bibitem{C03}
O. Christensen.
\newblock ``An introduction to frames and Riesz bases.''
\newblock Applied and Numerical Harmonic Analysis. {\em Birkh\"{a}user Boston, Inc., Boston, MA}, xxii+440 pp. (2003)

\bibitem{DBBA13}
H. B. Dang, K. Blanchfield, I. Bengtsson, and D. M. Appleby.
\newblock ``Linear dependencies in Weyl-Heisenberg orbits.''
\newblock {\em Quantum Inf. Process}, Vol. {\bfseries 12}, Issue 11, 3449--3475 (2013).

\bibitem{FKL09}
H. G. Feichtinger, W. Kozek, F. Luef.
\newblock ``Gabor analysis over finite abelian groups.''
\newblock {\em Appl. Comput. Harmon. Anal.} {\bfseries 26}(2), 230--248 (2009).

\bibitem{F09}
M. Fickus.
\newblock ``Maximally equiangular frames and Gauss sums''.
\newblock {\em J. Fourier Anal. Appl.} {\bfseries 15}, no. 3, 413--427 (2009).

\bibitem{FMT12}
M. Fickus, D. G. Mixon, J. C. Tremain.
\newblock ``Steiner equiangular tight frames.''
\newblock {\em Linear Algebra Appl.} {\bfseries 436}, no. 5, 1014--1027 (2012).

\bibitem{G46}
D. Gabor.
\newblock ``Theory of communication''.
\newblock {\em J. IEE, London} {\bfseries 93}(3), 429--457 (1946).


\bibitem{HRT96}
C. Heil, J. Ramanathan, and P. Topiwala.
\newblock ``Linear independence of time-frequency translates.''
\newblock {\em Proc. Amer. Math. Soc.} {\bfseries 124}(9), 2787--2795 (1996).

\bibitem{H07}
L. Hughston.
\newblock ``$d=3$ SIC-POVMs and Elliptic Curves.''
\newblock Perimeter Institute, Seminar Talk, available online at http://pirsa.org/07100040/ (2007)

\bibitem{IJM16}
J. W. Iverson, J. Jasper, D. G. Mixon.
\newblock ``Optimal line packings from nonabelian groups''.
\newblock {\em ArXiv preprint} \url{https://arxiv.org/abs/1609.09836} (2016).




\bibitem{JMF14}
J. Jasper, D. G. Mixon, M. Fickus.
\newblock ``Kirkman equiangular tight frames and codes.''
\newblock {\em IEEE Trans. Inform. Theory} {\bfseries 60}, no. 1, 170--181 (2014).

\bibitem{K15}
E. J. King.
\newblock ``Algebraic and geometric spread in finite frames.''
\newblock {\em Proc. SPIE 9597}, Wavelets and Sparsity XVI, 95970B (August 24, 2015); doi:10.1117/12.2188541.


\bibitem{KPR08}
F. Krahmer, G. E. Pfander, and P. Rashkov.
\newblock ``Uncertainty in time-frequency representations on finite abelian groups and applications.''
\newblock {\em Appl. Comput. Harmon. Anal.} {\bfseries 25}(2), 209--225 (2008).

\bibitem{LPW05}
J. Lawrence, G. E. Pfander, and D. Walnut.
\newblock ``Linear Independence of Gabor Systems in Finite Dimensional Vector Spaces.''
\newblock {\em J. Fourier Anal. Appl.}, {\bfseries 11}(6), 715--726 (2005).

\bibitem{M15}
R. D. Malikiosis.
\newblock ``A note on Gabor frames in finite dimensions.''
\newblock {\em Appl. Comput. Harmon. Anal.} {\bfseries 38}(2), 318--330 (2015).

\bibitem{Mes06}
R. Meshulam.
\newblock ``An uncertainty inequality for finite abelian groups''.
\newblock {\em European J. Combin.} {\bfseries 27}, no. 1, 63--67 (2006).

\bibitem{N99}
J. Neukirch.
\newblock ``Algebraic number theory.'' %Translated from the 1992 German original and with a note by Norbert Schappacher. With a foreword by G. Harder. 
\newblock Grundlehren der Mathematischen Wissenschaften, 322. {\em Springer-Verlag, Berlin}, xviii+571 pp. (1999)

\bibitem{P07}
G. E. Pfander.
\newblock ``Note on sparsity in signal recovery and in matrix identification''.
\newblock {\em Open Appl. Math. J.} {\bfseries 1}, 21--22 (2007). 

\bibitem{P13}
G. E. Pfander.
\newblock ``Gabor frames in finite dimensions.''
\newblock Finite frames (Chapter VI), 193--239, {\em Appl. Numer. Harmon. Anal.}, Birkh\"{a}user/Springer, New York (2013).

\bibitem{R05}
J. M. Renes.
\newblock ``Equiangular spherical codes in quantum cryptography''.
\newblock {\em Quantum Inf. Comput.} {\bfseries 5}, no. 1, 81--92 (2005).

\bibitem{S06}
A. J. Scott.
\newblock ``Tight informationally complete quantum measurements''.
\newblock {\em J. Phys. A} {\bfseries 39}, no. 43, 13507--13530 (2006).

\bibitem{S77}
J.-P. Serre.
\newblock ``Linear representations of finite groups.''
\newblock Translated from the second French edition by {\em Leonard L. Scott}. Graduate Texts in Mathematics, Vol. 42. {\em Springer-Verlag, New York-Heidelberg}, x+170 pp. (1977).

\bibitem{SH03}
T. Strohmer, R. W. Heath.
\newblock ``Grassmannian frames with applications to coding and communication''.
\newblock {\em Appl. Comput. Harmon. Anal.} {\bfseries 14}, no. 3, 257--275 (2003).

\bibitem{T05}
T. Tao.
\newblock ``An uncertainty principle for cyclic groups of prime order.''
\newblock {\em Math. Res. Lett.}, {\bfseries 12}:121--127 (2005).


\bibitem{T98}
V. Turaev.
\newblock ``Reciprocity for Gauss sums on finite abelian groups.''
\newblock {\em Math. Proc. Camb. Phil. Soc.}, {\bfseries 124}, 205--214 (1998).

\bibitem{W31}
H. Weyl.
\newblock ``The theory of groups and quantum mechanics.''
\newblock Translated by {\it H. P. Robertson}, XXII + 422 p. New York, {\em Dutton} (1931).

\bibitem{Z99}
G. Zauner.
\newblock ``Quantendesigns. Grundz\"{u}ge einer nichtkommutativen Designtheorie.''
\newblock PhD thesis, Univ. Wien (1999); English translation 
\newblock ``Quantum Designs: Foundations of a Non-commutative Design Theory.''
{\em Int. J. Quant. Inf.} {\bfseries 9}, no. 1, 445--507 (2011).








\end{thebibliography}
\end{document}